\documentclass{birkjour}

\usepackage[utf8]{inputenc}
\usepackage[T1]{fontenc}
\usepackage[english]{babel}
\usepackage{mathtools,amsmath,amssymb,amsfonts,amsthm,mathrsfs,latexsym}
\mathtoolsset{showonlyrefs}
\usepackage{enumitem}

\usepackage{graphicx,color}
\usepackage[hidelinks,pdftex]{hyperref}

\usepackage{geometry}
\newtheorem{theorem}{Theorem}[section]
\newtheorem{lemma}[theorem]{Lemma}
\newtheorem{proposition}[theorem]{Proposition}

\theoremstyle{definition}

\numberwithin{equation}{section}

\begin{document}

\title[Multiplicity of positive solutions of an indefinite superlinear problem]{A proof of a conjecture \\on the multiplicity of positive solutions \\of an indefinite superlinear problem}

\author[G.~Feltrin]{Guglielmo Feltrin}

\address{
Dipartimento di Scienze Matematiche, Informatiche e Fisiche\\
Universit\`{a} degli Studi di Udine\\
Via delle Scienze 206, 33100 Udine, Italy}

\email{guglielmo.feltrin@uniud.it}

\author[J.~L\'{o}pez-G\'{o}mez]{Juli\'{a}n L\'{o}pez-G\'{o}mez}

\address{
Instituto de Matem\'{a}tica Interdisciplinar\\
Departamento de An\'{a}lisis Matem\'{a}tico y Matem\'{a}tica Aplicada\\
Universidad Complutense de Madrid\\
Plaza de las Ciencias 3, 28040 Madrid, Spain}

\email{jlopezgo@ucm.es}

\author[J.~C.~Sampedro]{Juan Carlos Sampedro}

\address{
Instituto de Matem\'{a}tica Interdisciplinar \\ Departamento de Matem\'{a}tica Aplicada a la Ingenier\'{i}a Industrial\\
Universidad Polit\'{e}cnica de Madrid\\
Ronda de Valencia 3, 28012 Madrid, Spain}

\email{juancarlos.sampedro@upm.es}



\date{}

\begin{abstract}
This paper solves in a positive manner a conjecture stated in 2000 by R.~G\'omez-Re\~nasco and J.~L\'opez-G\'omez regarding the multiplicity of positive solutions of a paradigmatic class of superlinear indefinite boundary value problems.
\end{abstract}

\thanks{The first author acknowledges the Grup\-po Na\-zio\-na\-le per l'Anali\-si Ma\-te\-ma\-ti\-ca, la Pro\-ba\-bi\-li\-t\`{a} e le lo\-ro Appli\-ca\-zio\-ni (GNAMPA) of the Isti\-tu\-to Na\-zio\-na\-le di Al\-ta Ma\-te\-ma\-ti\-ca (INdAM) and the support of PRIN Project 20227HX33Z ``Pattern formation in nonlinear phenomena''. The second and third author are supported by the Ministry of Science and Innovation of Spain under Research Grant PDI2021-123343NB-I00 and the Institute of Interdisciplinary Mathematics of the Complutense University of Madrid.
\\
\textbf{Preprint -- January 2025}}

\maketitle

\section{Introduction}\label{section-1}

In this paper, we prove a result conjectured via numerical experiments in 2000 by G\'omez-Re\~nasco and L\'opez-G\'omez in \cite[p.~19]{GRLG-00} concerning the number of positive solutions of the boundary value problem
\begin{equation}\label{1.1}
\begin{cases}
\, -u''=\lambda u+a(x) u^{p} \quad \hbox{in}\;\; (0,L),
\\ \, u(0)=u(L)=0,
\end{cases}
\end{equation}
where $p>1$, $\lambda\in\mathbb{R}$ is regarded as a parameter, and $a(x)$ is a sign-changing weight function.
\par
The positive solutions of  \eqref{1.1} can be seen as the steady states of the parabolic problem
\begin{equation}\label{eq-parabolic}
\begin{cases}
\, \frac{\partial u}{\partial t} - \frac{\partial^2 u}{\partial x^2}  = \lambda u + a(x) u^{p},
&(x,t)\in (0,L) \times (0,+\infty),
\vspace{2pt}\\  
\, u(0,t) =u(L,t)= 0,
& t \in (0,+\infty), 
\\
\, u(\cdot,0) = u_{0} \geq 0, & \hbox{in}\;\; (0,L).
\end{cases}
\end{equation}
Diffusive problems of this type have been widely analyzed in Population Dynamics. More generally, the monographs of Cantrell and Cosner \cite{CC-03}, L\'{o}pez-G\'{o}mez \cite{LG-16}, Murray \cite{Mu-93} and Okubo \cite{Ok-80} discuss the following multidimensional counterpart of  \eqref{eq-parabolic}
\begin{equation}\label{eq-paramult}
\begin{cases}
\, \frac{\partial u}{\partial t} - \Delta u  = \lambda u + a(x) u^{p},
&(x,t)\in \Omega \times (0,+\infty),
\vspace{2pt}\\  
\, u(x,t) = 0,
& (x,t)\in\partial\Omega \times (0,+\infty),
\\
\, u(\cdot,0) = u_{0} \geq 0, & \hbox{in}\;\; \Omega,
\end{cases}
\end{equation}
where $\Omega$ is a smooth bounded subdomain of $\mathbb{R}^N$, for some integer $N\geq 1$. In these
models, $u(x,t)$ represents the density at time $t$, at the location $x\in\Omega$, of a species inhabiting the territory $\Omega$, $\lambda$ stands for its growth rate if $\lambda>0$, or its death rate
if $\lambda<0$, and $u_{0}$ is the initial population density. In nature $\lambda<0$, for example,
when pesticides are used in high concentrations, or a certain patch of the
environment is polluted by introducing chemicals, waste products or poisonous substances. Rather paradoxically, the harsher the environmental conditions the richer the dynamics of the species, as discussed by L\'{o}pez-G\'{o}mez, Molina-Meyer and Tellini in \cite{LGMMTe-14}.
\par
Dirichlet boundary conditions entail that the habitat edges, modeled by $\partial\Omega$,  are completely hostile for the individuals of the species $u$. In \eqref{eq-paramult}, the term $a(x) u^{p}$ represents the reaction component. Although in Mathematical Biology the coefficient $a(x)$ is typically non-positive, the fact that $a(x)$ can change sign in this paper reflects the existence of some facilitative environmental conditions within the patches where $a(x)>0$. Precisely, in  \eqref{eq-paramult}, the weight function $a(x)$ can represent a ``food source'' for the population which can be positive, zero or negative in the habitat $\Omega$ meaning that the environment is respectively favorable, neutral or unfavorable for the population. Naturally, the search of steady states of \eqref{eq-paramult} and the investigation of their stability is crucial in order to analyze extinction, persistence or coexistence for the underlying population.
\par
Starting from the Eighties, significant attention has been devoted to problems of the form \eqref{1.1}, covering the more general PDE
\begin{equation}
\label{eq-PDE-intro}
\, -\Delta u = \lambda u + a(x) g(u),
\end{equation}
for some smooth $g \colon[0,+\infty)\to [0,+\infty)$. These problems  are known as \textit{superlinear indefinite problems}. The term ``indefinite'' refers to the fact that the weight function $a(x)$ has no definite sign.
This terminology has become widely used starting from Hess and Kato \cite{HeKa-80}. The term ``superlinear'' describes the growth of the nonlinear term $g$ which is of superlinear type, i.e., $g(u) \sim u^{p}$ with $p>1$ as $u\uparrow \infty$. From the pioneering findings of Brown and Lin \cite{BrLi-80} and Hess and Kato \cite{HeKa-80} and the subsequent papers of Alama and Tarantello \cite{AlTa-93,AlTa-96}, Amann and
L\'{o}pez-G\'{o}mez \cite{AmLG-98}, Berestycki, Capuzzo-Dolcetta and Nirenberg \cite{BeCDNi-94,BeCDNi-95}, and L\'{o}pez-G\'{o}mez \cite{LG-97} many others have followed, dealing also with more general nonlinearities $g$ or other boundary value problems associated with \eqref{eq-PDE-intro}.
Recent literature indicates that this remains a very active area of research. The interested reader is sent to L\'{o}pez-G\'{o}mez and Tellini \cite{LGTe-14},  L\'{o}pez-G\'{o}mez, Tellini and Zanolin \cite{LGTeZa-14} and Tellini \cite{Te-20} for some results addressing \eqref{1.1}, and to the monographs of Feltrin \cite{Fe-18} and L\'{o}pez-G\'{o}mez \cite{LG-16} for a more complete list of references.
\par
Thanks to Amann and L\'{o}pez-G\'{o}mez \cite{AmLG-98}, it is well known that \eqref{1.1} possesses an unbounded component of positive solutions, $\mathscr{C}^+$, of $\mathbb{R} \times \mathcal{C}([0,L])$ such that $(\tfrac{\pi^{2}}{L^{2}},0)\in\overline{\mathscr{C}^+}$, i.e., $\mathscr{C}^+$ bifurcates from $u=0$ at $\lambda=\tfrac{\pi^{2}}{L^{2}}$. Moreover, \eqref{1.1} cannot admit a positive solution for sufficiently large $\lambda>\tfrac{\pi^{2}}{L^{2}}$. Furthermore, by the existence of universal a priori bounds uniform on compact subintervals of $\lambda\in\mathbb{R}$ for the positive solutions
of \eqref{1.1}, it is apparent that
\begin{equation*}
   (-\infty,\tfrac{\pi^{2}}{L^{2}})\subseteq \mathcal{P}_\lambda(\mathscr{C}^+),
\end{equation*}
where $\mathcal{P}_\lambda$ stands for the $\lambda$-projection operator defined by
\begin{equation*}
\mathcal{P}_\lambda (\lambda,u)=\lambda,
\quad (\lambda,u)\in \mathbb{R} \times \mathcal{C}([0,L]).
\end{equation*}
Actually, according to G\'{o}mez-Re\~{n}asco and L\'{o}pez-G\'{o}mez \cite{GRLG-00,GRLG-01}, either $\mathcal{P}_\lambda(\mathscr{C}^+)=(-\infty,\tfrac{\pi^{2}}{L^{2}})$, or there exists $\lambda_t>\tfrac{\pi^{2}}{L^{2}}$ such that $\mathcal{P}_\lambda(\mathscr{C}^+)=(-\infty,\lambda_t]$. Moreover, \eqref{1.1} admits some stable
positive solution if and only if $\lambda\in (\tfrac{\pi^{2}}{L^{2}},\lambda_t]$, and, in such a case, the stable solution is unique, and it equals the minimal positive solution of \eqref{1.1}. The fact that $\lambda_t$ is a ``turning point'' is emphasized by its subindex. By a recent result of Fern\'{a}ndez-Rinc\'{o}n and L\'{o}pez-G\'{o}mez \cite{FRLG-21}, the latest results require the nonlinearity to be of power type, i.e., there should exist a $p>1$ such that  $g(u)=u^p$ for all $u\geq 0$ in \eqref{eq-PDE-intro}.
\par
Throughout this paper, we assume that $a\in\mathcal{C}([0,L])$ and that
\begin{enumerate}[leftmargin=26pt,labelsep=8pt,label=\textup{$(\mathrm{H}a)$}]
\item for some integer $n\geq 1$, there are $2n+2$ points
\begin{equation*}
0=\tau_{0} \leq \sigma_{1}<\tau_{1}<\sigma_{2}<\tau_{2}<\cdots<\sigma_{n}<\tau_{n} \leq \sigma_{n+1}=L
\end{equation*}
such that
\label{hp-Ha}
\begin{align*}
&a(x)>0 \;\;\hbox{for all $x\in I^{+}_{i}\coloneqq(\sigma_{i},\tau_{i})$, $i=1,\dots, n$,}
\\
&a(x)<0 \;\; \hbox{for all $x\in I^{-}_{i}\coloneqq(\tau_{i},\sigma_{i+1})$, $i=0,\dots,n$.}
\end{align*}
\end{enumerate}
Condition \ref{hp-Ha} means that there are $n$ intervals where $a>0$ separated away by intervals where $a<0$.
Moreover, we require $a$ to satisfy the following growth condition:
\begin{enumerate}[leftmargin=26pt,labelsep=8pt,label=\textup{$(\mathrm{G})$}]
\item for every $i\in\{1,\ldots,n\}$, there exist $\rho_{i} \colon [\sigma_{i},\tau_{i}]\to\mathbb{R}$ continuous and bounded away from zero in a neighborhood of $\partial I^{+}_{i} = \{\sigma_{i},\tau_{i}\}$ and a constant $\gamma_i>0$ such that
\begin{equation*}
a(x) = \rho_{i}(x)[\mathrm{dist}(x,\partial I^{+}_{i})]^{\gamma_i} \quad \hbox{for all}\;\; x\in I^{+}_{i}.
\end{equation*}
\label{hp-G}
\end{enumerate}
By a positive solution of \eqref{1.1} we mean a function $u\in\mathcal{C}^2([0,L])$ with $u\gneq 0$ (i.e. $u\geq 0$ but $u\not =  0$), solving \eqref{1.1}. Since $u\gneq 0$,  by the Theorem of Cauchy--Lipschitz, $u'(0)>0$, $u'(L)<0$, and $u(x)>0$ for all $x\in (0,L)$.
\par
Based on a series of systematic numerical experiments, in \cite{GRLG-00} it was conjectured that, under these conditions, there is a $\lambda_c<\tfrac{\pi^{2}}{L^{2}}$ such that, for every $\lambda < \lambda_c$, the problem \eqref{1.1} possesses, at least, $2^{n}-1$ positive solutions: among them, $n$ with a single peak around each of the maxima of $a(x)$ in the intervals $I^+_i$, $\frac{n(n-1)}{2}$ with two peaks, and, in general, $\frac{n!}{j!(n-j)!}$ with $j$ peaks for every $j\in\{1,\ldots,n\}$. This conjecture led to the investigation of various types of multiplicity results for positive solutions in this kind of problems. As the main goal of this paper is to deliver a proof of this conjecture, we will describe the main lines of research that inspired our proof.
\par
Assume \ref{hp-Ha}, let $a^{\pm}(x)$ denote the positive/negative part of $a(x)$, and consider, instead of \eqref{1.1},  the following boundary value problem
\begin{equation}\label{eq-FZ}
\begin{cases}
\, -u'' = \left(a^+(x)-\mu \, \dfrac{a^-(x)}{\|a^-\|_{L^\infty}} \right)u^p \quad\hbox{in}\;\; (0,L),
\\
\, u(0) = u(L) = 0,
\end{cases}
\end{equation}	
where
\begin{equation*}
   \mu\equiv \|a^-\|_{L^\infty}\in(0,+\infty)
\end{equation*}
is regarded as a parameter. The problem \eqref{eq-FZ} is a particular case of \eqref{1.1}, with $\lambda =0$,  where the weight function $a(x)$ has been factorized within the spirit of L\'{o}pez-G\'{o}mez \cite{LG-97,LG-00}. Concerning \eqref{eq-FZ}, assuming that $a(x)$ satisfies \ref{hp-Ha} with $n=2$ or $n=3$,  the existence of, at least, $2^{2}-1=3$ and $2^{3}-1=7$ positive solutions, respectively,  for sufficiently large $\mu>0$ was established in Gaudenzi, Habets and Zanolin \cite{GaHaZa-03,GaHaZa-04} via a shooting technique. More recently, the existence of at least $2^{n}-1$ positive solutions of \eqref{eq-FZ} for arbitrary $n\geq 2$ has been proved in Feltrin and Zanolin \cite{FeZa-15} with a combinatorial argument involving the topological degree. Some multidimensional versions of these results were given in Bonheure, Gomez and Habets  \cite{BoGoHa-05} and in Gir\~{a}o and Gomes \cite{GiGo-09}, and some exact multiplicity results have been found in a  recent paper of Feltrin and Troestler, \cite{FeTr-pp}. A similar topological approach has been used by Feltrin, Sampedro and Zanolin  in \cite{FeSaZa-23} in dealing with a class of  planar superlinear systems.
\par
However, the results concerning \eqref{eq-FZ} do not solve the conjecture of G\'{o}mez-Re\~{n}asco and L\'{o}pez-G\'{o}mez \cite{GRLG-00} because the
roles of the parameters $\lambda$ and $\mu$ in \eqref{1.1} and in \eqref{eq-FZ} are substantially different, apart from being independent. Thus, after 25 years, the conjecture of \cite{GRLG-00} remains completely open. Our main aim in this paper is to prove this conjecture through the next result.

\begin{theorem}\label{th1.1}
Let $p>1$ and $a(x)$ satisfy \ref{hp-Ha} and \ref{hp-G}. Then, there exists $\lambda_{c}=\lambda_c(a)<0$ such that, for every $\lambda < \lambda_{c}$, the problem \eqref{1.1} possesses, at least, $2^{n}-1$ positive solutions.
\end{theorem}

Although there is a large literature on bump and multi-bump solutions for nonlinear
Schr\"{o}dinger equations, starting with Ambrosetti, Badiale and Cingolani \cite{AmBaCi-97}, del Pino and Felmer \cite{PiFe-97,PiFe-98}, Dancer and Wei \cite{DaWe-98}, Wei \cite{We-96},  Wang and Zeng \cite{WaZe-97}, up to, e.g., Byeon and Tanaka \cite{ByTa-13}, in most of these papers $a(x)$ is a positive function. As a consequence,  none of these results (mostly based on classical   techniques of the calculus of variations, like appropriate variants of the mountain--pass lemma of Ambrosetti and Rabinowitz \cite{AmRa-73})  can be applied without some additional substantial work to our general nodal setting.
\par
This paper is organized as follows. Section~\ref{section-2} presents a one-dimensional nonlinear Liouville theorem that is employed to provide a priori bounds for the non-negative solutions of \eqref{1.1}. In Section~\ref{section-3}, we present the topological degree setting and  combine it with these a priori bounds  to prove Theorem~\ref{th1.1}.
\par
Throughout this paper, for every pair of intervals,  $J\subseteq\mathbb{R}$ and  $K\subseteq J$, and any function $v\in L^{\infty}(J)$, we will set
\begin{equation*}
   \|v\|_{L^{\infty}(K)} \coloneqq \sup_{x\in K} |v(x)|, \qquad
   \|v\|_{L^{\infty}} \coloneqq \|v\|_{L^{\infty}(J)}.
\end{equation*}
\par

\section{A priori estimates}\label{section-2}

The aim of this section is to present some preliminary technical results, which we will subsequently exploit in Section~\ref{section-3} for the proof of Theorem~\ref{th1.1}. Some of them inherit a great interest on their
own.
\par
First, we provide a nonlinear Liouville-type theorem that will be useful for establishing a priori bounds for the non-negative solutions of \eqref{1.1}.

\begin{lemma}\label{leii.1}
Let $\alpha>0$, $\gamma\geq0$, $\kappa\geq0$, and $\beta\in\mathcal{C}([-\kappa,+\infty), [0,+\infty))$. Then, the boundary value problem
\begin{equation}\label{ii.1}
\begin{cases}
\, -u''=\alpha \,(x+\kappa)^{\gamma} \, u^{p} + \beta(x) \quad\hbox{in}\;\; [-\kappa,+\infty),
\\ \, u(0)=1, \\ \, 0\leq u \leq 1,
\end{cases}
\end{equation}
cannot admit a non-negative solution $u\in\mathcal{C}^{2}([-\kappa,+\infty))$.
\end{lemma}

\begin{proof}
Let $u\in\mathcal{C}^{2}([-\kappa,+\infty))$ be a non-negative solution of \eqref{ii.1}. As $0\leq u \leq 1$ and $u(0)=1$, necessarily $u'(0)\leq 0$. Moreover, as $u''\leq 0$, a direct integration yields $u'\leq 0$ in $[0,+\infty)$. Thus, two different cases can occur: either $u(x)>0$ for all $x\geq 0$, or there
exists $x_0>0$ such that
\begin{equation}
\label{ii.2}
  u(x)>0 \;\;\hbox{for all}\;\; x\in [0,x_0),\;\;\hbox{and}\;\; u(x)=0\;\; \hbox{for all}\;\; x\geq x_0.
\end{equation}
Suppose that \eqref{ii.2} holds for some $x_0>0$. Then, since
\begin{equation*}
-u''(x)\geq \alpha \, (x+\kappa)^{\gamma} \, u^{p}(x) > 0 \quad \hbox{for all}\;\; x\in (0,x_0),
\end{equation*}
it becomes apparent that $u''(x)<0$ for every $x\in (0,x_0)$. Consequently,
\begin{equation*}
u'(x_0) = u'(0) + \int_{0}^{x_0}u''(x)\,\mathrm{d}x < u'(0)\leq 0
\end{equation*}
and, hence, $u(x_0+\eta)<0$ for sufficiently small $\eta>0$, which contradicts the positivity of $u$.
Therefore,
\begin{equation*}
  u(x)>0 \quad \hbox{for all}\;\; x\in [0,+\infty).
\end{equation*}
Thus, $-u''(x)>0$ for all $x\geq 0$. So,
\begin{equation}
\label{ii.3}
  u'(x)<0\quad \hbox{for all}\;\; x>0.
\end{equation}
Moreover, for every $x\in[1,+\infty)$,
\begin{equation}
\label{ii.4}
-u''(x)\geq \alpha \, (1+\kappa)^{\gamma} \, u^{p}(x).
\end{equation}
Thus, multiplying \eqref{ii.4} by $u'\leq 0$ and integrating in $[1,x]$, we obtain that
\begin{equation*}
\frac{(u'(1))^{2}}{2}-\frac{(u'(x))^{2}}{2}
\leq \alpha \, (1+\kappa)^{\gamma} \left( \dfrac{u^{p+1}(x)}{p+1}
- \dfrac{u^{p+1}(1)}{p+1}\right) \quad \hbox{for all} \;\; x\in[1,+\infty).
\end{equation*}
Equivalently, for every $x\in[1,+\infty)$,
\begin{equation}
\label{new-2}
(u'(x))^{2} \geq (u'(1))^{2}+\frac{2\alpha \, (1+\kappa)^{\gamma}}{p+1}\bigl(u^{p+1}(1)-u^{p+1}(x)\bigr).
\end{equation}
Consequently, by \eqref{ii.3}, we can infer that, for every $x\geq 1$,
\begin{equation*}
\frac{-u'(x)}{\sqrt{(u'(1))^{2}+\frac{2\alpha \, (1+\kappa)^{\gamma}}{p+1}\bigl(u^{p+1}(1)-u^{p+1}(x)\bigr)}}\geq 1.
\end{equation*}
Actually, due to \eqref{ii.3}, a direct integration in $[1,x]$ and the change of variables
$y=u(x)$ yield
\begin{align*}
x-1 &\leq \int_{u(x)}^{u(1)}\frac{\mathrm{d}y}{\sqrt{(u'(1))^{2}+\frac{2\alpha \, (1+\kappa)^{\gamma}}{p+1}\bigl(u^{p+1}(1)-y^{p+1}\bigr)}}
\\ &\leq \int_{0}^{u(1)} \frac{\mathrm{d}y}{-u'(1)}=
\frac{u(1)}{-u'(1)} \quad \hbox{for all}\;\; x\in[1,+\infty),
\end{align*}
which implies
\begin{equation*}
  x\leq 1 +\frac{u(1)}{-u'(1)}.
\end{equation*}
So, giving a contradiction for sufficiently large $x>1$. This ends the proof.
\end{proof}

The next result provides us with lower bounds for the $L^{\infty}$-norm of the non-trivial non-negative solutions of a homotopy problem associated with \eqref{1.1}.

\begin{lemma}\label{leii.2}
For every $\lambda< 0$, there exists $r_{\lambda}>0$ such that every non-negative solution $u\neq 0$ of
\begin{equation}
\label{ii.5}
\begin{cases}
\, -u'' = \vartheta(\lambda u+a(x) u^{p}) \quad \hbox{in}\;\; (0,L),
\\ \, u(0)=u(L)=0,
\end{cases}
\end{equation}
satisfies $\|u\|_{L^{\infty}} > r_{\lambda}$, regardless the value of $\vartheta\in (0,1]$.
Moreover, one can make the choice
\begin{equation*}
r_{\lambda}\coloneqq\left( \frac{-\lambda}{\|a\|_{L^{\infty}}} \right)^{\!\frac{1}{p-1}}.
\end{equation*}
If $\vartheta=0$, then $u=0$ is the unique solution of \eqref{ii.5}.
\end{lemma}
\begin{proof}
Let $u\neq 0$ be a non-negative solution of \eqref{ii.5} for some $\vartheta\in (0,1]$, and let $x_{0}\in (0,L)$ be such that
\begin{equation*}
   u(x_{0})=\|u\|_{L^{\infty}}>0.
\end{equation*}
Then, $u''(x_{0})\leq 0$ and hence,
\begin{equation*}
0\leq -u''(x_{0})
= \vartheta \bigl(\lambda \|u\|_{L^{\infty}}+a(x_{0})\|u\|_{L^{\infty}}^{p}\bigr)
= \vartheta \|u\|_{L^{\infty}} \bigl(\lambda+a(x_{0})\|u\|_{L^{\infty}}^{p-1}\bigr).
\end{equation*}
Thus,
\begin{equation*}
\lambda +a(x_{0})\|u\|_{L^{\infty}}^{p-1}\geq 0.
\end{equation*}
Consequently, for every $\lambda<0$, we find that
\begin{equation*}
a(x_{0})\|u\|_{L^{\infty}}^{p-1}\geq -\lambda > 0.
\end{equation*}
So, $a(x_0)>0$, which entails $x_{0}\in I^{+}_{i}$ for some $i\in \{1,\cdots,n\}$. Therefore,
\begin{equation*}
\|u\|_{L^{\infty}}\geq \left( \frac{-\lambda}{a(x_{0})} \right)^{\!\frac{1}{p-1}}\geq \left( \frac{-\lambda}{\|a\|_{L^{\infty}}} \right)^{\!\frac{1}{p-1}}\equiv r_\lambda.
\end{equation*}
This concludes the proof.
\end{proof}

Next, we will provide with a priori bounds for the $L^{\infty}$-norm of the non-negative solutions of another  homotopy problem associated with \eqref{1.1}.

\begin{lemma}\label{leii.3}
Let $\mathcal{I}\subseteq \{1,\ldots,n\}$ be an arbitrary non-empty subset of indices, and let $w\in\mathcal{C}([0,L])$ be such that
\begin{equation}\label{ii.6}
w(x)>0\;\;\hbox{for all} \;\; x\in \bigcup_{i\in\mathcal{I}}I^{+}_{i},\qquad w\equiv 0 \;\;\hbox{in}\;\;
 [0,L]\setminus \bigcup_{i\in\mathcal{I}}\overline{I^{+}_{i}},
\end{equation}
and, for every $i\in\mathcal{I}$, there exists a function $q_{i} \colon \overline{I_{i}^{+}}\to\mathbb{R}$ continuous and bounded away from zero in a neighborhood of $\partial I^{+}_{i} = \{\sigma_{i},\tau_{i}\}$
for which
\begin{equation*}
w(x) = q_{i}(x) [\mathrm{dist}(x,\partial I^{+}_{i})]^{\gamma_i} \quad \hbox{for all}\;\; x\in I^{+}_{i},
\end{equation*}
where $\gamma_i>0$ is the constant of condition \ref{hp-G} (thanks to \eqref{ii.6}, we can take $q_i=0$
if $i\notin\mathcal{I}$). Then, for every $\lambda<0$, there exists a constant $R_{\lambda}>0$ such that:
\begin{enumerate}
\item[{\rm (i)}] every non-negative solution $u$ of
\begin{equation}\label{ii.7}
\begin{cases}
\, -u''=\lambda u+a(x) u^{p}+\mu w \quad \hbox{in}\;\; (0,L), \\
\, u(0)=u(L)=0,
\end{cases}
\end{equation}
satisfies $\|u\|_{L^{\infty}}<R_{\lambda}$, regardless the value of $\mu\geq 0$;
\item[{\rm (ii)}] there exists $\mu_*=\mu_{*}(\lambda)>0$ such that \eqref{ii.7} does not
admit a non-negative solution $u\neq 0$ if $\mu\geq \mu_*$.
\end{enumerate}
\end{lemma}

\begin{proof}
Let $\mathcal{I}\subseteq \{1,\ldots,n\}$ with $\mathcal{I}\neq\emptyset$, and pick $\lambda<0$ and $w$ be as in the statement. Suppose, by contradiction, that (i) is false.
Then, there are two sequences $\{\mu_{n}\}_{n\in\mathbb{N}}\subset [0,+\infty)$ and $\{u_{n}\}_{n\in\mathbb{N}}\subset \mathcal{C}^{2}([0,L])$, with $u_n\gneq 0$, such that, for every $n\in\mathbb{N}$,
\begin{equation}
\label{ii.8}
\begin{cases}
\, -u_{n}''=\lambda u_{n}+a(x) u_{n}^{p}+\mu_{n} w \quad \hbox{in}\;\; (0,L),
\\ \, u_{n}(0)=u_{n}(L)=0,
\end{cases}
\end{equation}
and, for some sequence $\{x_{n}\}_{n\in\mathbb{N}}\subset (0,L)$,
\begin{equation}
\label{ii.9}
0<u_{n}(x_{n})=\|u_{n}\|_{L^{\infty}}=:M_{n}\to +\infty \quad \hbox{as}\;\; n\to+\infty.
\end{equation}
Evaluating the equation in \eqref{ii.8} at the maximum point $x=x_{n}$, we find from \eqref{ii.9} that, for every  $n\in\mathbb{N}$,
\begin{equation}
\label{ii.10}
\lambda M_{n}+a(x_{n})M_{n}^{p} + \mu_{n} w(x_{n}) = - u_{n}''(x_{n}) \geq 0.
\end{equation}
Suppose that $a(x_n)\leq 0$ for some $n\in\mathbb{N}$. Then,  as $\lambda<0$ and, due to \eqref{ii.6},  $w(x_n)=0$, we have that
\begin{equation*}
   \lambda M_{n}+a(x_n)M_{n}^{p} < 0,
\end{equation*}
which contradicts \eqref{ii.10}. Therefore, $a(x_{n})>0$ for all $n\in\mathbb{N}$. Equivalently, by \ref{hp-Ha},
\begin{equation*}
\{x_{n}\}_{n\in \mathbb{N}}\subset \bigcup_{i=1}^{n}I^{+}_{i}.
\end{equation*}
By compactness, there exists an index $j_{0}\in\{1,2,\ldots,n\}$ such that
\begin{equation*}
   \lim_{n\to +\infty} x_{n}= x_{0} \;\; \hbox{for some}\;\; x_{0}\in\overline{I^{+}_{j_{0}}}=
   [\sigma_{j_0},\tau_{j_0}] \eqqcolon [\alpha,\beta].
\end{equation*}
Moreover, due to \ref{hp-Ha}, by choosing an appropriate subsequence, we can suppose that
\begin{equation*}
   x_n \in I_{j_{0}}^{+} \quad \hbox{for all}\;\; n\in\mathbb{N}.
\end{equation*}
Subsequently, we will distinguish two different cases according to the location of $x_0$
in the compact interval $[\alpha,\beta]$.

\medskip

\noindent
\textit{Case 1.} Suppose that $x_{0}\in(\alpha,\beta)$. Then, for every $n\in\mathbb{N}$, we consider  the scaling functions
\begin{equation}\label{ii.11}
v_{n}\left(\nu_{n}^{-1}(x-x_{n})\right):=\nu_{n}^{\frac{2}{p-1}}u_{n}(x) \;\;\hbox{for all}\;\; x\in [\alpha,\beta], \quad \hbox{where}\;\; \nu_{n}\coloneqq M_{n}^{\frac{1-p}{2}}.
\end{equation}
Observe that, thanks to \eqref{ii.9} and \eqref{ii.11},
\begin{equation}
\label{ii.12}
\lim_{n\to +\infty}\nu_{n}=0,
\end{equation}
because $p>1$, and that, particularizing \eqref{ii.11} at $x=x_n$, \eqref{ii.9} implies that
\begin{equation*}
v_{n}(0)=M_{n}^{-1}u_{n}(x_{n})=1 \quad \hbox{for all}\;\; n\in\mathbb{N}.
\end{equation*}
According to \eqref{ii.11}, for every $n\in\mathbb{N}$, the function $v_{n}$ is well defined in the interval \begin{equation}\label{ii.13}
J_{n}\coloneqq \Bigl[M_{n}^{\frac{p-1}{2}}(\alpha-x_{n}), M_{n}^{\frac{p-1}{2}}(\beta-x_{n})\Bigr],
\end{equation}
and $v_n$ can be equivalently defined by
\begin{equation*}
v_{n}(y)\coloneqq\nu_{n}^{\frac{2}{p-1}}u_{n}(x_{n}+\nu_{n} y) \quad \hbox{for all}\;\; y\in J_{n}.
\end{equation*}
By construction,
\begin{equation*}
\|v_n\|_{L^\infty(J_n)}=v_n(0)=1\quad \hbox{for all}\;\;n\in \mathbb{N}.
\end{equation*}
After some straightforward manipulations, it is easily seen that $v_n \in \mathcal{C}^{2}(J_n)$ and that
\begin{equation*}
-v_{n}''(y)=\lambda\nu_{n}^{2}v_{n}(y)+a(x_{n}+\nu_{n}y)v_{n}^{p}(y)+\nu_{n}^{\frac{2p}{p-1}}
\mu_{n}w(x_{n}+\nu_{n}y) \quad \hbox{for all}\;\; y\in J_{n}.
\end{equation*}
Next, we will estimate the limit
\begin{equation*}
\lim_{n\to+\infty} \Big( \nu_{n}^{\frac{2p}{p-1}}\mu_{n}\Big) =
\lim_{n\to+\infty}\frac{\mu_{n}}{M_{n}^{p}}.
\end{equation*}
Let $\varphi(x)\coloneqq \sin \frac{\pi x}{L}$ be the unique positive eigenfunction of the eigenvalue problem
\begin{equation}\label{ii.14}
\begin{cases}
\, -\varphi''=\Sigma_{1} \varphi \quad \hbox{in}\;\; (0,L),
\\ \, \varphi(0)=\varphi(L)=0,
\end{cases}
\end{equation}
with $\|\varphi\|_{L^{\infty}}=1$, where
\begin{equation*}
   \Sigma_1=\left(\frac{\pi}{L}\right)^{\!2}.
\end{equation*}
Then, multiplying the differential equation of \eqref{ii.8} by $\varphi$ and integrating by parts in $(0,L)$, we find that
\begin{align*}
\mu_{n} \int_{0}^{L}w\varphi & = \int_{0}^{L}\bigl[(\Sigma_{1}-\lambda)u_{n}\varphi-a(x)u_{n}^{p}\varphi\bigr]
\\ & \leq (\Sigma_{1}-\lambda)LM_{n} +\|a\|_{L^{\infty}}LM_{n}^{p}.
\end{align*}
Thus, for every $n\in\mathbb{N}$,
\begin{equation*}
\mu_{n}\leq \frac{(\Sigma_{1}-\lambda)LM_{n} +\|a\|_{L^{\infty}}LM_{n}^{p}}{\int_{0}^{L}w\varphi}.
\end{equation*}
Therefore, owing to \eqref{ii.9}, it becomes apparent that
\begin{equation}\label{ii.15}
\rho\coloneqq \limsup_{n\to+\infty}\frac{\mu_{n}}{M_{n}^{p}}\leq \frac{L \|a\|_{L^{\infty}}}{\int_{0}^{L}w\varphi}.
\end{equation}
Actually, by choosing an appropriate subsequence, we can assume that
\begin{equation}
\label{ii.16}
\lim_{n\to+\infty}\Big( \nu_{n}^{\frac{2p}{p-1}}\mu_{n}\Big) = \lim_{n\to+\infty}\frac{\mu_{n}}{M_{n}^{p}}=\rho.
\end{equation}
As $x_{0}\in(\alpha,\beta)$, it follows from \eqref{ii.9} and \eqref{ii.13} that $J_{n}\uparrow \mathbb{R}$ as $n\to+\infty$.  Thus, for every $R>0$, there is ${n}_{R}\in\mathbb{N}$ such that $[-R,R]\subseteq J_{n}$ for all $n\geq {n}_{R}$. Moreover, $v_{n}\in \mathcal{C}^{2}([ -R,R])$ satisfies
\begin{equation}\label{ii.17}
-v_{n}''(y)=\lambda\nu_{n}^{2}v_{n}(y)+a(x_{n}+\nu_{n}y)v_{n}^{p}(y)+
\nu_{n}^{\frac{2p}{p-1}}\mu_{n}w(x_{n}+\nu_{n}y), \quad y\in [-R,R].
\end{equation}
As $0\leq v_{n}\leq 1$, it follows from \eqref{ii.12}, \eqref{ii.16}, and \eqref{ii.17} that, for every  $\varepsilon>0$, there exists $n_0=n_0(\varepsilon)\in\mathbb{N}$ such that
\begin{align*}
\|v_{n}''\|_{L^\infty([-R,R])}
&\leq |\lambda|\nu_{n}^{2} \|v_{n}\|_{L^{\infty}}+\|a\|_{L^{\infty}} \|v_{n}\|_{L^{\infty}}^{p}+\nu_{n}^{\frac{2p}{p-1}}\mu_{n} \|w\|_{L^{\infty}}
\\ &\leq |\lambda|+\|a\|_{L^{\infty}}+(\rho+\varepsilon) \|w\|_{L^{\infty}}
\end{align*}
for all $n\geq n_0$. Consequently, since $v_n(0)=1$ and, hence, $v_n'(0) \leq 0$ for all $n\in\mathbb{N}$, we also have that
\begin{equation*}
 v'_n(y)=v'_n(0)+\int_0^y v''_n(s)\,\mathrm{d}s \leq \int_0^y v''_n(s)\,\mathrm{d}s
\quad \hbox{for all} \;\; y \in [-R,R].
\end{equation*}
Thus,
\begin{equation*}
\|v_n'\|_{L^\infty([-R,R])}\leq R \|v_n''\|_{L^\infty([-R,R])}
\leq R \bigl[ |\lambda|+\|a\|_{L^{\infty}}+(\rho+\varepsilon) \|w\|_{L^{\infty}} \bigr].
\end{equation*}
These estimates show that $\{v_n\}_{n\geq 1}$ is bounded in $\mathcal{C}^2([-R,R])$. Thus, by the Ascoli--Arzel\`a Theorem, we can extract a subsequence, labeled again by $n$, such that, for some $v\in \mathcal{C}^1([-R,R])$,
\begin{equation*}
\lim_{n\to +\infty}\|v_n-v\|_{\mathcal{C}^1([-R,R])}=0.
\end{equation*}
Actually, letting $n\to +\infty$ in \eqref{ii.17}, we also find that $v\in \mathcal{C}^2([-R,R])$
and that
\begin{equation}\label{ii.18}
-v''= a(x_0)v^p +\rho w(x_0)\quad \hbox{in}\; \; [-R,R].
\end{equation}
By a diagonal process based on a sequence of values of $R$, $\{R_n\}_{n\geq 1}$, such that $R_n\to+\infty$ as $n\to+\infty$, it becomes apparent that $v$ can be prolonged to $\mathbb{R}$ still satisfying \eqref{ii.18}. Therefore, $v\in\mathcal{C}^2(\mathbb{R})$ and it satisfies
\begin{equation*}
\begin{cases}
\, -v''=a(x_{0})v^{p}+\rho w(x_{0})\quad \hbox{in}\;\; \mathbb{R},
\\ \, v(0)=1, \quad 0 \leq v \leq 1,
\end{cases}
\end{equation*}
which contradicts Lemma~\ref{leii.1}, by choosing  $\alpha=a(x_0)>0$, $\gamma=0$ and $\beta=0$, and ends the proof in this case.

\medskip

\noindent
\textit{Case 2.}
Suppose that $x_{0}\in\{\alpha,\beta\}$. Without loss of generality, we can assume that $x_{0}=\alpha$,
as the proof in the case $x_0=\beta$ is analogous. In this case, we set $\gamma \coloneqq \gamma_{j_0}$, as defined in \ref{hp-G}, and, for every $n\in\mathbb{N}$, we consider the scaling function
\begin{equation}
\label{ii.19}
v_{n}\left(\nu_{n}^{-1}(x-x_{n})\right)=\nu_{n}^{\frac{2+\gamma}{p-1}}u_{n}(x)
\quad \hbox{for all}\;\; x\in [\alpha,\beta],\quad \hbox{where} \;\; \nu_{n}\coloneqq M_{n}^{\frac{1-p}{2+\gamma}}.
\end{equation}
Observe that
\begin{equation*}
\lim_{n\to +\infty}\nu_{n}=0, \qquad v_{n}(0)=M_{n}^{-1}u_{n}(x_{n})=1 \quad \hbox{for all}\;\; n\in\mathbb{N}.
\end{equation*}
According to \eqref{ii.19}, for every $n$, the function $v_{n}$ is well defined in the interval \begin{equation*}
\tilde J_{n}\coloneqq \Bigl[M_{n}^{\frac{p-1}{2+\gamma}}(\alpha-x_{n}), M_{n}^{\frac{p-1}{2+\gamma}}(\beta-x_{n})\Bigr].
\end{equation*}
Moreover,
\begin{equation}
\label{ii.20}
v_{n}(y)\coloneqq\nu_{n}^{\frac{2+\gamma}{p-1}}u_{n}(x_{n}+\nu_{n} y) \quad \hbox{for all}\;\; y\in \tilde J_{n}.
\end{equation}
Differentiating \eqref{ii.20} with respect to $y$ and substituting the result in the differential equation of
\eqref{ii.8}, after some straightforward manipulations, it follows from \ref{hp-G} that $v_n \in \mathcal{C}^{2}(\tilde J_n)$ and that it satisfies
\begin{align*}
-v_{n}''(y) =\lambda\nu_{n}^{2}v_{n}(y) & +\nu_{n}^{-\gamma} \rho_{j_{0}}(x_{n}  +\nu_{n}y)\bigl[\mathrm{dist}(x_{n}+\nu_{n}y,\partial I^{+}_{j_{0}})\bigr]^{\gamma} v_{n}^{p}(y) \\ & +\nu_{n}^{\frac{2p+\gamma}{p-1}}
\mu_{n}q_{j_{0}}(x_{n}+\nu_{n}y)\bigl[\mathrm{dist}(x_{n}+\nu_{n}y,\partial I^{+}_{j_{0}})\bigr]^{\gamma}
\end{align*}
for all $y\in \tilde J_{n}$. It should be remembered that $q_{j_{0}} \equiv 0$ if $j_0\notin \mathcal{I}$.
\par
Up to a subsequence, we can assume that $x_{n} \in [\alpha, \tfrac{\alpha+\beta}{2}]$ for all $n\in\mathbb{N}$. Moreover, by definition of the distance function,
\begin{equation*}
\mathrm{dist}(x,\partial I^{+}_{j_{0}}) =
\begin{cases}
\, x-\alpha & \hbox{if}\;\; x\in [\alpha, \tfrac{\alpha+\beta}{2}],
\\ \, \beta -x & \hbox{if} \;\; x\in [\tfrac{\alpha+\beta}{2},\beta].
\end{cases}
\end{equation*}
Thus,  for every $n\in\mathbb{N}$ and $y\in\tilde J_n$, we have that
\begin{align*}
-v_{n}''(y)
  =\lambda\nu_{n}^{2}v_{n}(y) & + \rho_{j_{0}}(x_{n}+\nu_{n}y)\bigl[M_{n}^{\frac{p-1}{2+\gamma}}(x_{n}-\alpha)+y\bigr]^{\gamma} v_{n}^{p}(y)
\\ & +\nu_{n}^{\frac{p(\gamma+2)}{p-1}}
\mu_{n} q_{j_{0}}(x_{n}+\nu_{n}y) \bigl[M_{n}^{\frac{p-1}{2+\gamma}}(x_{n}-\alpha)+y\bigr]^{\gamma}.
\end{align*}
Repeating the proof of \eqref{ii.15}, it becomes apparent that, also in this case,
\begin{equation*}
\rho
\coloneqq
\limsup_{n\to+\infty} \Big( \nu_{n}^{\frac{p(\gamma+2)}{p-1}} \mu_{n}\Big)
= \limsup_{n\to+\infty} \frac{\mu_{n}}{M_{n}^{p}}\leq \frac{L \|a\|_{L^{\infty}}}{\int_{0}^{L}w\varphi}.
\end{equation*}
Subsequently, we divide the proof into three parts according to the behavior of the limit
\begin{equation*}
\kappa \equiv \limsup_{n\to+\infty}\Big( M_{n}^{\frac{p-1}{2+\gamma}}(x_{n}-\alpha)\Big) \geq 0.
\end{equation*}
Suppose $\kappa = 0$. Then, $\tilde J_{n}$ approximates $[0,+\infty)$ as $n\to+\infty$. Moreover,
adapting the argument of the proof of Case 1, one can infer the existence of $v\in \mathcal{C}^{2}([0,+\infty))$ such that
\begin{equation*}
\begin{cases}
\, -v''(y)= \rho_{j_{0}}(\alpha)y^{\gamma}v^{p}(y)+\rho q_{j_{0}}(\alpha) y^{\gamma} \quad \hbox{for all} \;\; y\in [0,+\infty), \\ \, v(0)=1, \quad 0 \leq v \leq 1.
\end{cases}
\end{equation*}
According to Lemma~\ref{leii.1}, this is not possible. Therefore, $\kappa \in (0,+\infty]$.
\par
Suppose that  $\kappa = +\infty$. Then, $\tilde J_{n}$ approximates $\mathbb{R}$ as $n\to+\infty$. In this case, we perform the change of variable
\begin{equation*}
v_{n}(y):=w_{n}\bigl(\kappa_{n}^{\frac{\gamma}{2}}y\bigr),
\quad \hbox{where}\;\; \kappa_{n}\coloneqq M_{n}^{\frac{p-1}{2+\gamma}}(x_{n}-\alpha)\to \kappa=+\infty
 \;\;  \hbox{as}\;\; n\to +\infty.
\end{equation*}
Then, setting $z:= \kappa_{n}^{\frac{\gamma}{2}}y$, we have that
\begin{equation*}
  w_n(z)= v_n( \kappa_{n}^{-\frac{\gamma}{2}}y)
\end{equation*}
and $w_{n}$ satisfies the differential equation
\begin{align*}
-w_{n}''(z)  = \lambda \nu_{n}^{2} \kappa_{n}^{-\gamma} w_{n}(z) & + \rho_{j_{0}}(x_{n}+\nu_{n}\kappa_{n}^{-\frac{\gamma}{2}}z)\kappa_{n}^{-\gamma}(\kappa_{n}^{-\frac{\gamma}{2}}z
+\kappa_{n})^{\gamma} w_{n}^{p}(z) \\ &+\nu_{n}^{\frac{p(\gamma+2)}{p-1}}
\mu_{n}q_{j_{0}}(x_{n}+\nu_{n}\kappa_{n}^{-\frac{\gamma}{2}}z)
\kappa_{n}^{-\gamma}(\kappa_{n}^{-\frac{\gamma}{2}}z+\kappa_{n})^{\gamma}.
\end{align*}
Letting $n\to+\infty$ and arguing as in the previous case, one can infer the existence of $w\in \mathcal{C}^{2}(\mathbb{R})$ such that
\begin{equation*}
\begin{cases}
\, -w''= \rho_{j_{0}}(\alpha)w^{p}+\rho\, q_{j_{0}}(\alpha) \quad \hbox{in}\;\; \mathbb{R},
\\ \, w(0)=1, \quad 0 \leq w \leq 1.
\end{cases}
\end{equation*}
As, according to Lemma~\ref{leii.1}, such an $w$ cannot exist, we can infer that $\kappa \in (0,+\infty)$.
Thus, $\tilde J_{n}\uparrow [-\kappa,+\infty)$ as $n\to+\infty$, and, arguing as in the first case, there
 exists $v\in \mathcal{C}^{2}([-\kappa,\infty))$ such that
\begin{equation*}
\begin{cases}
\, -v''= \rho_{j_{0}}(\alpha)(y+\kappa)^{\gamma}v^{p}+\rho\, q_{j_{0}}(\alpha) (y+\kappa)^{\gamma}
\quad \hbox{in}\;\; [-\kappa, +\infty), \\ \, v(0)=1, \quad 0\leq v \leq 1.
\end{cases}
\end{equation*}
As, once again, Lemma~\ref{leii.1} guarantees that such an $v$ cannot exist,
this contradiction concludes the proof of Part (i).
\par
To prove Part (ii), let $\varphi(x)\coloneqq \sin \frac{\pi x}{L}$ be the unique positive eigenfunction of the eigenvalue problem \eqref{ii.14} such that $\|\varphi\|_{L^{\infty}}=1$. Then, multiplying the differential equation of \eqref{ii.7} by $\varphi$ and integrating by parts, we find that, for every positive solution
of \eqref{ii.7},
\begin{align*}
\mu \int_{0}^{L}w\varphi & = \int_{0}^{L}\left[(\Sigma_{1}-\lambda)u\varphi-a(x)u^{p}\varphi\right]
\\ & \leq \left[(\Sigma_{1}-\lambda) +\|a\|_{L^{\infty}}\|u\|_{L^{\infty}}^{p-1}\right] L \|u\|_{L^{\infty}}.
\end{align*}
Thus, according to Part (i), it is apparent that
\begin{equation*}
\mu < \mu_{*}(\lambda) :=  \frac{\left[(\Sigma_{1}-\lambda)+\|a\|_{L^{\infty}} R_{\lambda}^{p-1}\right] L R_{\lambda}}{\int_{0}^{L}w\varphi}.
\end{equation*}
Therefore, \eqref{ii.7} cannot admit non-negative solutions if $\mu\geq \mu_{*}(\lambda)$. This ends the proof.
\end{proof}

The next result provides us with the behavior of the non-negative solution of \eqref{1.1} in the intervals $I^{-}_{i}$ as $\lambda\downarrow-\infty$. It is Theorem~3.1 of Fencl and L\'{o}pez-G\'{o}mez \cite{FeLG-22}.

\begin{lemma}\label{leii.4}
For any given $i\in\{1,\ldots,n\}$ and $\delta>0$, and every compact subinterval $K\subseteq I^{-}_{i}$, there exists $\tilde{\lambda}_{i}<0$ such that
\begin{equation*}
\|u\|_{L^{\infty}(K)}<\delta
\end{equation*}
for all non-negative solution $u$ of \eqref{1.1} with $\lambda<\tilde{\lambda}_{i}$.
\end{lemma}

We conclude this section with a technical result establishing that, for sufficiently large $|\lambda|$, with $\lambda<0$, the $L^\infty$-norms of the positive solutions of \eqref{1.1} must be either small or large, but cannot take intermediate values.

\begin{lemma}\label{leii.5}
For any given $\rho>0$, there exists $\hat{\lambda}(\rho)<0$ such that, for every $\lambda<\hat{\lambda}(\rho)$ and any non-negative solution $u$ of \eqref{1.1},
\begin{equation*}
\|u\|_{L^{\infty}(I^{+}_{i})} \neq \rho \quad \hbox{for all} \;\; i\in\{1,\ldots,n\}.
\end{equation*}
\end{lemma}

\begin{proof}
Let $\rho>0$. Let $u\in \mathcal{C}^{2}([0,L])$ be a non-negative solution of \eqref{1.1} and fix an arbitrary $i\in\{1,\ldots,n\}$. Let us start by assuming that there is $x_{0}\in I^{+}_{i}$ such that
\begin{equation*}
u(x_{0})=\|u\|_{L^{\infty}(I^{+}_{i})}.
\end{equation*}
Then, since $x_0$ is an interior maximum, we have that
\begin{equation*}
  u'(x_{0})=0, \quad u''(x_{0})\leq 0.
\end{equation*}
Thus, it follows from \eqref{1.1} that
\begin{equation*}
0 \leq \lambda \|u\|_{L^{\infty}(I^{+}_{i})} + \|a\|_{L^{\infty}}\|u\|_{L^{\infty}(I^{+}_{i})}^{p}.
\end{equation*}
Hence,
\begin{equation}
\label{ii.21}
\|u\|_{L^{\infty}(I^{+}_{i})}\geq \left(\frac{-\lambda}{\|a\|_{L^{\infty}}}\right)^{\!\frac{1}{p-1}}\equiv r_\lambda.
\end{equation}
As $\lim_{\lambda\downarrow-\infty}r_\lambda=+\infty$, there exists $\lambda_i<0$ such that $r_\lambda > \rho$ for all $\lambda < \lambda_i$. Therefore, \eqref{ii.21} implies that
\begin{equation*}
\|u\|_{L^{\infty}(I^{+}_{i})} > \rho \quad \hbox{for all} \;\; \lambda<\lambda_i.
\end{equation*}
Subsequently, we assume that the maximum of $u$ in $\overline{I^{+}_{i}}$ is attained on its boundary, in $x_{0}\in \partial I^{+}_{i}$. Suppose, in addition, that $x_{0}\in \{0,L\}$. Then,
\begin{equation*}
  u(x_0)=u'(x_0)=0
\end{equation*}
and, hence, $u\equiv 0$ in $I^{+}_{i}$. Therefore, $u\equiv 0$ in $[0,L]$, which entails
\begin{equation*}
 \|u\|_{L^{\infty}(I^{+}_{i})} =0 < \rho \quad \hbox{for all} \;\; i\in\{1,\ldots,n\}.
\end{equation*}
When $x_0\neq 0, L$, $x_0$ must be a point dividing the positivity interval $I^{+}_{i}=(\sigma_i,\tau_i)$ and some adjacent negativity interval, say for instance $I^{-}_{i-1}=(\tau_{i-1},\sigma_i)$. Then, $x_{0}=\sigma_{i}$. Pick a sufficiently small $\eta>0$ such that
\begin{equation*}
K_{i}\coloneqq[\tau_{i-1}+\eta, x_0-\eta]\subset I^{-}_{i-1}.
\end{equation*}
By Lemma~\ref{leii.4}, there exists $\tilde \lambda_i<0$ such that, for every solution $u\geq 0$ of \eqref{1.1},
\begin{equation}
\label{ii.22}
u(x_0-\eta) < \rho\quad \hbox{for all} \;\; \lambda < \tilde \lambda_i.
\end{equation}
Suppose that there exist $\lambda_{0}<\tilde \lambda_i$ and a solution $u\in\mathcal{C}^{2}([0,L])$ of \eqref{1.1} with $\lambda=\lambda_{0}$ such that
\begin{equation}
\label{ii.23}
u(x_0)=u(\sigma_i)=\|u\|_{L^{\infty}(I^{+}_{i})} \geq \rho.
\end{equation}
Since $a\leq 0$ in $\overline{I_i^-}$, it follows from \eqref{1.1} that
\begin{equation*}
-u''=\lambda_{0} u +a(x)u^p \leq 0 \quad \hbox{in}\;\; I_i^-.
\end{equation*}
Thus, $u''\geq 0$ in $\overline{I_i^-}$. Actually,
\begin{equation*}
  u''(x_0)=-\lambda_0 u(x_0) > 0.
\end{equation*}
Therefore, \eqref{ii.22} and \eqref{ii.23} imply that $u'(x_0)>0$, which contradicts
the fact that $u$ attaints at $x_0$ its maximum in $\overline{I_i^+}$. This contradiction shows that,
\begin{equation*}
\|u\|_{L^{\infty}(I^{+}_{i})} < \rho \quad \hbox{for all} \;\; \lambda<\tilde \lambda_i.
\end{equation*}
Consequently, by choosing $\hat \lambda_i\coloneqq\min\{\lambda_i,\tilde \lambda_i\}$, we have proved that, for every non-negative solution $u$ of \eqref{1.1},
\begin{equation*}
\|u\|_{L^{\infty}(I^{+}_{i})} \neq \rho \quad \hbox{for all} \;\; \lambda<\hat \lambda_i.
\end{equation*}
By taking
\begin{equation*}
\hat{\lambda}(\rho)\coloneqq\min\bigl{\{} \hat \lambda_i \colon i\in\{1,\ldots,n\} \bigr{\}},
\end{equation*}
the proof is completed.
\end{proof}

\section{Proof of Theorem~\ref{th1.1}}\label{section-3}

The main goal of this section is to deliver the proof of Theorem~\ref{th1.1}. First, we will introduce
a topological setting for studying problem \eqref{1.1}. Our approach is based on the Leray--Schauder degree. Accordingly, we first transform the boundary value problem \eqref{1.1} into a fixed point equation in a Banach space. Next, we prove the existence of fixed points by detecting some domains where we can compute the topological degree and showing that it is different from zero. This leads readily to the proof of Theorem~\ref{th1.1}.

\subsection{Topological degree setting}\label{section-3.1}

Let us consider the boundary value problem
\begin{equation}\label{iii.1}
\begin{cases}
\, -u''=\lambda u^{+}+a(x) (u^{+})^{p} \quad \hbox{in}\;\; (0,L),
\\ \, u(0)=u(L)=0,
\end{cases}
\end{equation}
where $u^{+}$ denotes the positive part of the function $u$, that is,
\begin{equation*}
   u^{+}\coloneqq\max\{u,0\}.
\end{equation*}
By the maximum principle (see, e.g., Feltrin and Zanolin \cite[Lem.~2.1]{FeZa-15}), every solution of \eqref{iii.1} is non-negative in $[0,L]$ and thus solves \eqref{1.1}. Moreover, if $u\neq 0$ is a non-negative solution of \eqref{iii.1}, then $u$ is a strongly positive solution of \eqref{1.1}, in the sense that  $u(x)>0$ for all $x\in (0,L)$,
$u'(0)>0$, and $u'(L)<0$.
\par
Our first goal is to transform \eqref{iii.1} into a fixed point equation in a suitable Banach space.
To this end, we introduce the linear operator $K \colon \mathcal{C}([0,L])\to \mathcal{C}^2([0,L])$ defined, for every $f \in \mathcal{C}([0,L])$, by
\begin{equation*}
(K f)(x)\coloneqq \int_0^L \mathscr{K}(x,y) f(y) \,\mathrm{d}y, \quad x\in[0,L],
\end{equation*}
where $\mathscr{K}\in\mathcal{C}([0,L]^2)$ is the kernel given by
\begin{equation*}
\mathscr{K}(x,y) \coloneqq
\begin{cases}
\, y(L-x) & \hbox{if}\;\; y\leq x,
\\ \, x(L-y) & \hbox{if}\;\; y>x.
\end{cases}
\end{equation*}
By the definition of the integral operator $K$, for every $f \in \mathcal{C}([0,L])$, the function $u\coloneqq Kf$ is the unique solution of the linear boundary value problem
\begin{equation*}
\begin{cases}
\, -u''= f \quad \hbox{in}\;\; [0,L], \\ \, u(0)=u(L)=0.
\end{cases}
\end{equation*}
Subsequently, for every integer $k\geq0$, we denote by $\mathcal{C}^k_0([0,L])$ the closed subspace of
the Banach space $\mathcal{C}^{k}([0,L])$ consisting of all functions $u\in \mathcal{C}^{k}([0,L])$ such that
$u(0)=u(L)=0$.
It is easily seen that $K \colon \mathcal{C}([0,L])\to \mathcal{C}_0^2([0,L])$ is linear and continuous.
Next, we consider the canonical injection $\jmath \colon \mathcal{C}^2_0([0,L]) \hookrightarrow \mathcal{C}([0,L])$.
Thanks to the Ascoli--Arzel\`{a} Theorem, $\jmath$ is a linear compact operator. Thus,
\begin{equation*}
\mathcal{K} \coloneqq \jmath {K} \colon \mathcal{C}([0,L])\to\mathcal{C} ([0,L])
\end{equation*}
also is a linear compact operator. By construction, the problem \eqref{iii.1} can be expressed as a fixed point equation. Indeed, $u$ solves \eqref{iii.1} if and only if
\begin{equation}
\label{iii.2}
u = \mathcal{K} (\lambda u^{+}+ a (u^{+})^p).
\end{equation}
Note that, in terms of  the nonlinear operator $\Phi_{\lambda} \colon \mathcal{C}([0,L])\to\mathcal{C}([0,L])$ defined by
\begin{equation*}
(\Phi_{\lambda} u)(x)\coloneqq\int_{0}^{L}\mathscr{K}(x,y) \bigl[\lambda u^{+}(y)+ a(y) (u^{+}(y))^p\bigr] \, \mathrm{d}y, \quad x\in [0,L],
\end{equation*}
\eqref{iii.2} can be equivalently expressed as
\begin{equation}
\label{iii.3}
  u = \Phi_\lambda u.
\end{equation}
By construction, solving \eqref{1.1} is equivalent to solve \eqref{iii.3}. Clearly, $\Phi_{\lambda}$ is completely continuous in $\mathcal{C}([0,L])$ endowed with the maximum norm. Therefore, for any given open and bounded subset $\mathcal{O} \subset \mathcal{C}([0,L])$ such that
\begin{equation*}
\Phi_{\lambda} u \neq u \quad \hbox{for all} \;\; u\in\partial\mathcal{O},
\end{equation*}
the Leray--Schauder degree $\mathrm{deg}_{\mathrm{LS}}(I-\Phi_{\lambda}, \mathcal{O})$ is well defined.
To look for fixed points of $\Phi_{\lambda}$ (i.e., non-negative solutions of \eqref{1.1}), we should
detect appropriate domains $\mathcal{O}$ such that
\begin{equation*}
    \mathrm{deg}_{\mathrm{LS}}(I-\Phi_{\lambda}, \mathcal{O})\neq0.
\end{equation*}
In order to do this we will use the additivity/excision, normalization and homotopy invariance properties of the Leray--Schauder degree, as discussed, e.g., by Fonseca and Gangbo in \cite{FoGa-95}.

\subsection{The sets $\Omega^{\mathcal{I}}$ and the computation of the degree}\label{section-3.2}

Let us fix an arbitrary constant $\rho>0$. By Lemma~\ref{leii.5}, there exists $\hat{\lambda}(\rho)<0$ such that, for every $\lambda < \hat{\lambda}(\rho)$, every non-negative solution of \eqref{1.1} satisfies
\begin{equation*}
\|u\|_{L^{\infty}(I^{+}_{i})} \neq \rho \quad \hbox{for all}\;\; i\in\{1,\ldots,n\}.
\end{equation*}
For every $\lambda<0$, let $0<r_{\lambda}<R_{\lambda}$ be the constants whose existence is guaranteed
by Lemmas~\ref{leii.2} and \ref{leii.3}, respectively. Since
\begin{equation*}
\lim_{\lambda\to-\infty}r_{\lambda}=\lim_{\lambda\to -\infty} \left(\frac{-\lambda}{\|a\|_{L^{\infty}}}\right)^{\!\frac{1}{p-1}} = +\infty,
\end{equation*}
there exists $\lambda_{*}<\hat{\lambda}(\rho)$  such that
\begin{equation*}
\rho < r_{\lambda} \quad \hbox{for all} \;\; \lambda < \lambda_{*},
\end{equation*}
where $\hat{\lambda}(\rho)$ is the constant whose existence is guaranteed by Lemma \ref{leii.5}.
\par
With all these constants in mind, for any given set of indexes $\mathcal{I} \subseteq\{1,\ldots,n\}$ and $\lambda<\lambda_*$, we can introduce the following open and bounded set
\begin{equation*}
\Omega^{\mathcal{I}}_{\lambda}
\coloneqq
\left\{u\in\mathcal{C}([0,L]) \colon \|u\|_{L^{\infty}} < R_{\lambda}, \; \|u\|_{L^{\infty}(I^{+}_{i})} < \rho \;\; \hbox{if}\;\; i \in \{1,\dots,n\}\setminus\mathcal{I} \right\}.
\end{equation*}
In particular,
\begin{equation*}
\Omega^\emptyset_{\lambda}
\coloneqq
\left\{u\in\mathcal{C}([0,L]) \colon \|u\|_{L^{\infty}} < R_{\lambda}, \; \|u\|_{L^{\infty}(I^{+}_{i})} < \rho \;\; \hbox{for all}\;\; i \in \{1,\dots,n\}\right\}.
\end{equation*}
The next results  exploit the lemmas of Section~\ref{section-2} to compute the degrees of the operators $I-\Phi_{\lambda}$ in these sets $\Omega^{\mathcal{I}}_{\lambda}$.

\begin{proposition}\label{pr3.1}
For every $\lambda<\lambda_{*}$, it holds that
\begin{equation*}
\mathrm{deg}_{\mathrm{LS}}(I-\Phi_{\lambda}, \Omega^{\emptyset}_\lambda)=1.
\end{equation*}
\end{proposition}
\begin{proof}
Consider the parameter-dependent boundary value problem \eqref{ii.5}, where $\vartheta\in[0,1]$.
First, we observe that, since $\lambda<\lambda_*<0$,  every non-negative solution $u\in\Omega^{\emptyset}_\lambda$ of \eqref{ii.5} is convex in each $I_{i}^{-}$,  and so $\|u\|_{L^\infty}$ is reached in some interval $I_{i}^{+}$. Thus, $\|u\|_{L^\infty}<\rho$. Equivalently, $u\in B_\rho(0)$; here $B_\rho(0)$ denotes the open ball of radius $\rho$ centered at the origin in $\mathcal{C}([0,L])$.  Note that $B_{\rho}(0) \subset \Omega^{\emptyset}_\lambda$, because, by construction,
\begin{equation*}
   0<\rho<r_\lambda<R_\lambda\quad \hbox{if}\;\; \lambda<\lambda_*.
\end{equation*}
Subsequently, we consider the homotopy operator
\begin{equation*}
H \colon [0,1]\times \mathcal{C}([0,L]) \to \mathcal{C}([0,L]), \quad H(\vartheta, u)\coloneqq \vartheta \Phi_{\lambda}(u).
\end{equation*}
By the definition of $\Phi_\lambda$, it becomes apparent that $u$ is a fixed point of the completely continuous operator $H(\vartheta, \cdot)$, for some $\vartheta\in[0,1]$, if and only if $u$ is a non-negative solution of \eqref{ii.5}. Since $\rho<r_{\lambda}$, by Lemma~\ref{leii.2}, we have that
\begin{equation*}
\vartheta\Phi_{\lambda}u\neq u \quad \hbox{for all}
 \;\; \vartheta\in[0,1]\;\;\hbox{and}\;\; u\in \partial B_\rho(0).
\end{equation*}
As a consequence, the degree $\mathrm{deg}(I-H(\vartheta,\cdot),\Omega^{\emptyset}_{\lambda})$ is well defined for every $\vartheta\in[0,1]$, and, by homotopy invariance, it is constant with respect to $\vartheta$. Therefore,
\begin{align*}
\mathrm{deg}_{\mathrm{LS}}(I-\Phi_{\lambda},\Omega^{\emptyset}_{\lambda})
&= \mathrm{deg}_{\mathrm{LS}}(I-H(1,\cdot),\Omega^{\emptyset}_{\lambda})\\ &
= \mathrm{deg}_{\mathrm{LS}}(I-H(0,\cdot),\Omega^{\emptyset}_{\lambda})
= \mathrm{deg}_{\mathrm{LS}}(I,\Omega^{\emptyset}_{\lambda})\\ &
= \mathrm{deg}_{\mathrm{LS}}(I,B_\rho(0)) = 1.
\end{align*}
where the last two equalities follow from the excision and normalization properties of the degree, respectively. The proof is complete.
\end{proof}

\begin{proposition}\label{pr3.2}
For every $\lambda<\lambda_{*}$ and any non-empty subset $\mathcal{I}\subseteq\{1,\dots,n\}$, it holds that
\begin{equation*}
\mathrm{deg}_{\mathrm{LS}}(I-\Phi_{\lambda}, \Omega^{\mathcal{I}}_{\lambda}) = 0.
\end{equation*}
\end{proposition}

\begin{proof}
Fix any $\mathcal{I}\subseteq\{1,\dots,n\}$ with $\mathcal{I}\neq\emptyset$, and let $w\in\mathcal{C}([0,L])$ be a weight function satisfying \eqref{ii.6} and the remaining properties of the statement of Lemma \ref{leii.3}. Now, consider the parameter-dependent boundary value problem \eqref{ii.7}, where $\mu\in[0,\mu_{*}]$, with $\mu_{*}>0$ given by Lemma~\ref{leii.3} (ii).
\par
According to \eqref{ii.6}, $w\equiv0$ on each  $I_{i}^{-}$. Thus, since $\lambda<\lambda_*<0$, every non-negative solution $u\in\Omega^{\mathcal{I}}_{\lambda}$ of \eqref{ii.7} is convex in each $I_i^-$.
Hence, its maximum must be reached in some interval $I_{i}^{+}$.
Recalling that $w\equiv0$ in $I_{i}^{+}$ if $i\in\{1,\ldots,n\}\setminus\mathcal{I}$, it follows from
Lemma~\ref{leii.3} that
\begin{equation}
\label{iii.4}
   \|u\|_{L^{\infty}(I_{i}^{+})} < R_{\lambda} \quad \hbox{if}\;\; i\in\mathcal{I}.
\end{equation}
Moreover, by the definition of $\Omega_\lambda^\mathcal{I}$,
\begin{equation}
\label{iii.5}
\|u\|_{L^{\infty}(I_{i}^{+})} < \rho \quad \hbox{if}\;\; i\in\{1,\ldots,n\}\setminus\mathcal{I}.
\end{equation}
Throughout the rest of this proof, we consider the homotopy operator
\begin{equation*}
H \colon [0,\mu^{*}]\times \mathcal{C}([0,L]) \to \mathcal{C}([0,L]), \quad H(\mu, u)\coloneqq
 \Phi_{\lambda}u + \mu w.
\end{equation*}
By construction, we already know that, for every $\mu\in[0,\mu^{*}]$,  $u$ is a fixed point of the completely continuous operator $H(\mu, \cdot)$ if and only if $u$ is a non-negative solution of \eqref{ii.7}. Moreover,
according to \eqref{iii.4} and \eqref{iii.5}, it becomes apparent that
\begin{equation*}
H(\mu, u)\neq u \quad \hbox{for all} \;\; \mu\in[0,\mu^{*}]\;\; \hbox{and}\;\; u\in \partial \Omega^{\mathcal{I}}_{\lambda}.
\end{equation*}
In particular, the degree $\mathrm{deg}(I-H(\mu,\cdot),\Omega^{\mathcal{I}}_{\lambda})$ is well defined for every $\mu\in[0,\mu^{*}]$ and, by homotopy invariance, it is constant with respect to $\mu$. Therefore,
\begin{align*}
\mathrm{deg}_{\mathrm{LS}}(I-\Phi_{\lambda},\Omega^{\mathcal{I}}_{\lambda})
&= \mathrm{deg}_{\mathrm{LS}}(I-H(0,\cdot),\Omega^{\mathcal{I}}_{\lambda})
\\
&= \mathrm{deg}_{\mathrm{LS}}(I-H(\mu^{*},\cdot),\Omega^{\mathcal{I}}_{\lambda})
= 0,
\end{align*}
where the last equality follows from the fact that, thanks to Lemma~\ref{leii.3} (ii),
\begin{equation*}
   H(\mu^{*},u)=u
\end{equation*}
has no solution. This ends the proof.
\end{proof}

\subsection{Completing the proof of Theorem~\ref{th1.1}}\label{section-3.3}

To show the existence of, at least, $2^{n}-1$ positive solutions for \eqref{1.1}, it is sufficient to construct $2^{n}-1$ disjoint open and bounded subsets of $\mathcal{C}([0,L])$, not containing $u=0$, such that the Leray--Schauder degree of $I-\Phi_{\lambda}$ is different from zero on any of these sets. These properties guarantee the existence of, at least,  $2^{n}-1$ non-trivial fixed points of $\Phi_{\lambda}$,
which provide us with $2^{n}-1$ positive solutions of \eqref{1.1}. Our construction proceeds as follows.
For any given $\mathcal{I}\subseteq\{1,\ldots,n\}$, we consider the open and bounded set
\begin{equation*}
\Lambda^{\mathcal{I}}_{\lambda}\coloneqq
\left\{ u \in \mathcal{C}([0,L]) \colon
\begin{array}{l}
\|u\|_{L^{\infty}(0,L)} < R_{\lambda},
\vspace{1pt}\\
\|u\|_{L^{\infty}(I_{i}^{+})} > \rho \;\; \hbox{if}\;\; i\in\mathcal{I},
\vspace{1pt}\\
\|u\|_{L^{\infty}(I_{i}^{+})} < \rho\; \; \hbox{if}\;\; i\in\{1,\ldots,n\}\setminus\mathcal{I}
\end{array} \right\}.
\end{equation*}
The following properties hold from the definition of the $\Lambda^{\mathcal{I}}_{\lambda}$'s:
\begin{itemize}
\item since the number of nonempty subsets of a set with $n$ elements is
\begin{equation*}
\sum_{i=0}^{n}\binom{n}{i} = 2^n,
\end{equation*}
the cardinality of $\{\Lambda^{\mathcal{I}}_{\lambda} \colon \mathcal{I}\subseteq\{1,\ldots,n\}\}$ is $2^{n}$;
\item $0\in\Lambda^{\mathcal{I}}_{\lambda}$ if and only if $\mathcal{I}=\emptyset$;
\item $\Lambda^{\mathcal{J}}_{\lambda}\cap \Lambda^{\mathcal{K}}_{\lambda} = \emptyset$ if  $\mathcal{J}\neq \mathcal{K}$;
\item for every subset $\mathcal{I}\subseteq\{1,\ldots,n\}$,
\begin{equation}
\label{3.vi}
\Omega^{\mathcal{I}}_{\lambda} = \bigcup_{\mathcal{J}\subseteq \mathcal{I}}\Lambda^{\mathcal{J}}_{\lambda} \cup \bigcup_{i\in\mathcal{I}} \left\{ u \in \mathcal{C}([0,L]) \colon \|u\|_{L^{\infty}(0,L)} < R_{\lambda}, \; \|u\|_{L^{\infty}(I_{i}^{+})} = \rho \right\}.
\end{equation}
\end{itemize}

The next result is the key step to conclude the proof of Theorem~\ref{th1.1}.

\begin{proposition}\label{pr3.3}
For every  subset $\mathcal{I}\subseteq\{1,\dots,n\}$,
\begin{equation}\label{3.vii}
\mathrm{deg}_{\mathrm{LS}}(I-\Phi_{\lambda},\Lambda^{\mathcal{I}}_{\lambda}) = (-1)^{|\mathcal{I}|}
\quad \hbox{for all}\;\; \lambda<\lambda_{*},
\end{equation}
where $|\mathcal{I}|$ stands for the cardinal of $\mathcal{I}$, and $\lambda_*$ satisfies the prerequisites
of Section~\ref{section-3.2}; in particular, $\rho<r_\lambda$ if $\lambda<\lambda_*$.
\end{proposition}

\begin{proof}
Although the proof follows the same steps as the one of Feltrin and Zanolin \cite[Lem.~4.1]{FeZa-15}, by the sake of completeness, since \eqref{1.1} differs from the problem of \cite{FeZa-15}, we will again deliver complete technical details here.
\par
Let $\lambda<\lambda_{*}$. Then, by Propositions \ref{pr3.1} and \ref{pr3.2}, we have that, for every subset $\mathcal{I}\subseteq\{1,\dots,n\}$,
\begin{equation}
\label{3.viii}
\mathrm{deg}_{\mathrm{LS}}(I-\Phi_{\lambda},\Omega_{\lambda}^{\mathcal{I}}) 
=
\begin{cases}
\, 0 &\hbox{if}\;\;  \mathcal{I}\neq \emptyset,
\\
\, 1 &\hbox{if}\;\;  \mathcal{I} = \emptyset.
\end{cases}
\end{equation}
Subsequently, we denote by $\mathcal{I}$ an arbitrary subset of $\{1,\dots,n\}$.
To prove  \eqref{3.vii} we will use an inductive argument on the cardinal of $\mathcal{I}$, $|\mathcal{I}|$. To do this, for every integer $k$ with $0\leq k \leq |\mathcal{I}|$, we introduce the auxiliary property
\begin{equation*}
\mathscr{P}(k) \;:\;\; \mathrm{deg}_{\mathrm{LS}}(I-\Phi_{\lambda},\Lambda_{\lambda}^{\mathcal{J}})=
(-1)^{|\mathcal{J}|}\;\; \hbox{for each}\;\; \mathcal{J}\subseteq\mathcal{I}\;\; \hbox{with, at most, $k$ elements.}
\end{equation*}
If we prove $\mathcal{P}(|\mathcal{I}|)$, then \eqref{3.vii} follows.
\medskip
\par
\noindent
\textit{Verification of $\mathscr{P}(0)$.}
If $\mathcal{I}=\emptyset$, then $\Omega_{\lambda}^{\emptyset} = \Lambda_{\lambda}^{\emptyset}$ and, since $|\emptyset|=0$, it is apparent that  \eqref{3.vii} follows from \eqref{3.viii}.
\par
\medskip
\noindent
\textit{Verification of $\mathscr{P}(1)$.}
For $\mathcal{J}=\emptyset$, the result has just been proven. Suppose $\mathcal{J}=\{j\}$, with $j\in\mathcal{I}$. Then, by \eqref{3.vi}, we have that
\begin{equation*}
\Omega^{\{j\}}_{\lambda} = \Lambda^{\emptyset}_{\lambda} \cup \Lambda^{\{j\}}_{\lambda} \cup  \left\{ u \in \mathcal{C}([0,L]) \colon \|u\|_{L^{\infty}(0,L)} < R_{\lambda}, \;\;\|u\|_{L^{\infty}(I_{j}^{+})} = \rho \right\}.
\end{equation*}
According to Lemma~\ref{leii.5}, there is no solution of $\Phi_{\lambda}u=u$ with $\|u\|_{L^{\infty}(I_{j}^{+})} = \rho$. Thus, by the additivity/excision property of the degree, we can infer from \eqref{3.viii} that
\begin{align*}
\mathrm{deg}_{\mathrm{LS}}(I-\Phi_{\lambda},\Lambda^{\mathcal{J}}_{\lambda}) & = \mathrm{deg}_{\mathrm{LS}}(I-\Phi_{\lambda},\Lambda^{\{j\}}_{\lambda}) = \mathrm{deg}_{\mathrm{LS}}(I-\Phi_{\lambda},\Omega^{\{j\}}_{\lambda}\setminus \Lambda^{\emptyset}_{\lambda}) \\
& = \mathrm{deg}_{\mathrm{LS}}(I-\Phi_{\lambda},\Omega^{\{j\}}_{\lambda}) - \mathrm{deg}_{\mathrm{LS}}(I-\Phi_{\lambda}, \Lambda^{\emptyset}_{\lambda})
\\
&= 0 -1 = -1 = (-1)^{|\mathcal{J}|}.
\end{align*}
\par
\medskip
\noindent
\textit{Verification of the implication $\mathscr{P}(k-1) \Rightarrow \mathscr{P}(k)$, for all $1\leq k \leq |\mathcal{I}|$.} Assume that $\mathscr{P}(k-1)$ holds for some $1\leq k \leq |\mathcal{I}|$. Then,
\begin{equation}
\label{3.ix}
  \mathrm{deg}_{\mathrm{LS}}(I-\Phi_{\lambda},\Lambda_{\lambda}^{\mathcal{J}})=
(-1)^{|\mathcal{J}|}
\end{equation}
for every subset $\mathcal{J}\subset \mathcal{I}$ having, at most, $k-1$ elements. Thus, to prove
$\mathscr{P}(k)$, it suffices to show that \eqref{3.ix} also holds for an arbitrary subset $\mathcal{J}\subseteq\mathcal{I}$ with $|\mathcal{J}|=k$. Pick one of these $\mathcal{J}$'s. Then, owing to
\eqref{3.vi}, we have that
\begin{equation*}
\Omega^{\mathcal{J}}_{\lambda} = \Lambda^{\mathcal{J}}_{\lambda}\cup\bigcup_{\mathcal{K}\subsetneq \mathcal{J}} \Lambda^{\mathcal{K}}_{\lambda} \cup
\bigcup_{i\in\mathcal{I}} \left\{ u \in \mathcal{C}([0,L]) \colon \;\|u\|_{L^{\infty}(0,L)} < R_{\lambda}, \;\;  \|u\|_{L^{\infty}(I_{i}^{+})} = \rho \right\}.
\end{equation*}
Moreover, by Lemma~\ref{leii.5}, the equation $\Phi_{\lambda}u=u$ cannot admit a fixed point
\begin{equation*}
    u\in \bigcup_{i\in\mathcal{I}} \left\{ u \in \mathcal{C}([0,L]) \colon\; \|u\|_{L^{\infty}(0,L)} < R_{\lambda},\;\; \|u\|_{L^{\infty}(I_{i}^{+})} = \rho \right\}.
\end{equation*}
Thus, by the additivity and existence properties of the degree,
\begin{equation*}
\mathrm{deg}_{\mathrm{LS}}(I-\Phi_{\lambda}, \Lambda^{\mathcal{J}}_{\lambda})  = \mathrm{deg}_{\mathrm{LS}}(I-\Phi_{\lambda}, \Omega^{\mathcal{J}}_{\lambda}) - \sum_{\mathcal{K}\subsetneq \mathcal{J}} \mathrm{deg}_{\mathrm{LS}}(I-\Phi_{\lambda}, \Lambda^{\mathcal{K}}_{\lambda}).
\end{equation*}
Consequently, thanks to \eqref{3.viii} and $\mathscr{P}(k-1)$, we find that
\begin{equation}
\label{3.10}
\mathrm{deg}_{\mathrm{LS}}(I-\Phi_{\lambda}, \Lambda^{\mathcal{J}}_{\lambda})
= 0 - \sum_{\mathcal{K}\subsetneq \mathcal{J}}(-1)^{|\mathcal{K}|} = - \sum_{\mathcal{K}\subseteq \mathcal{J}}(-1)^{|\mathcal{K}|} + (-1)^{|\mathcal{J}|}.
\end{equation}
Finally, note that
\begin{equation*}
\sum_{\mathcal{K}\subseteq \mathcal{J}}(-1)^{|\mathcal{K}|} = 0,
\end{equation*}
because in any finite set there are so many subsets of even cardinality as there are with odd cardinality.
 Therefore, \eqref{3.10} implies that
\begin{equation*}
\mathrm{deg}_{\mathrm{LS}}(I-\Phi_{\lambda}, \Lambda^{\mathcal{J}}_{\lambda}) = (-1)^{|\mathcal{J}|},
\end{equation*}
and, hence, $\mathscr{P}(k)$ is proved. In conclusion, $\mathcal{P}(|\mathcal{I}|)$ is proved and
\eqref{3.vii} holds.
\end{proof}

We can now conclude the proof of Theorem~\ref{th1.1}. According to Proposition \ref{pr3.3}, we have that
\begin{equation*}
\mathrm{deg}_{\mathrm{LS}}(I-\Phi_{\lambda},\Lambda^{\mathcal{I}}_{\lambda})=\pm 1\neq 0 \quad
\hbox{for all} \;\; \mathcal{I}\subseteq \{1,\dots,n\}.
\end{equation*}
Thus, taking into account that $0\notin \Lambda_\lambda^\mathcal{I}$ if $\mathcal{I}\neq \emptyset$, and
that
\begin{equation*}
   \Lambda^{\mathcal{J}}_{\lambda}\cap \Lambda^{\mathcal{K}}_{\lambda} = \emptyset\quad
   \hbox{if}\;\; \mathcal{J}\neq \mathcal{K},
\end{equation*}
by the existence property of the degree, we can deduce the existence of, at least, $2^n-1$ non-trivial solutions of the operator equation $u=\Phi_{\lambda}(u)$; at least one in each $\Lambda^{\mathcal{I}}_{\lambda}$ with $\mathcal{I}\neq\emptyset$. We already know that
these solutions must be strongly positive solutions of \eqref{1.1}. This ends the proof of Theorem~\ref{th1.1}.
\qed

\section*{Declarations}

\medskip

\subsubsection*{Authors' contributions}
All the authors equally contributed to the implementation of the research and to the writing of the manuscript.

\subsubsection*{Availability of data and materials}
No datasets were generated or analysed during the current study.

\subsubsection*{Conflict of interest}
The authors declare no competing interests.

\bibliographystyle{elsart-num-sort.bst}
\bibliography{FLGS-biblio}

\end{document}